\documentclass[12pt]{amsart}

\usepackage{amsmath, amsthm, amscd, amssymb, amsfonts, amsxtra, amssymb, latexsym}
\usepackage{enumerate}
\usepackage{verbatim}
\usepackage{tgpagella}
\usepackage[T1]{fontenc} 
\usepackage{t1enc}
\usepackage[utf8x]{inputenc}
\usepackage{makecell}
\usepackage{hyperref}
\usepackage{marginnote}
\hypersetup{colorlinks = true,	allcolors  = blue}

\usepackage{soul}
\usepackage[normalem]{ulem}

\newcommand{\na}{\mathbb{N}}

\newcommand{\rr}{\mathcal{L}}
\newcommand{\G}{\mathcal{G}}
\newcommand{\ff}{\mathbb{F}}
\newcommand{\pl}{\mathbb{P}}

\newcommand{\cod}{\mathcal{C}}

\newcommand{\dif}{\operatorname{Diff}}

\newcommand{\con}{\operatorname{Con}}
\renewcommand{\cot}{\operatorname{Cotr}}
\newcommand{\sop}{\operatorname{Supp}}
\newcommand{\tr}{\operatorname{Tr}}

\newcommand{\gal}{\operatorname{Gal}}
\newcommand{\divisor}{\operatorname{Div}}

\newcommand{\sk}{\smallskip}
\newcommand{\msk}{\medskip}

\newtheorem{thm}{Theorem}[section]
\newtheorem{prop}[thm]{Proposition}
\newtheorem{lem}[thm]{Lemma}
\newtheorem{coro}[thm]{Corollary}
\theoremstyle{definition}
\newtheorem{rem}[thm]{Remark}
\newtheorem{exam}[thm]{Example}
\newtheorem{defi}[thm]{Definition}

\theoremstyle{remark}

\usepackage{color}
\def\blue{\color{blue}}   

\newcommand{\podesta}[1]{\textcolor{blue}{#1}}

\begin{document}
\numberwithin{equation}{section}
\title{The conorm code of an AG-code}
\author[M. Chara, R. Podest\'a, R. Toledano]{Mar\'ia Chara, Ricardo A.\@ Podest\'a, Ricardo Toledano}
\dedicatory{\today}
\thanks{Partially supported by CONICET, UNL CAI+D 2016, SECyT-UNC, CSIC}

\address{Mar\'ia Chara, FIQ -- CONICET, Universidad Nacional del Litoral, Santiago del Estero 2829, (3000) Santa Fe, Argentina. {\it E-mail: mchara@santafe-conicet.gov.ar}}

\address{Ricardo A.\@ Podest\'a, FaMAF -- CIEM (CONICET), Universidad Nacional de Córdoba, Av.\@ Me\-di\-na Allende 2144, 
	(5000) C\'ordoba, Argentina. {\it E-mail: podesta@famaf.unc.edu.ar}}

\address{Ricardo Toledano, Facultad de Ingeniería Química, Universidad Nacional del Litoral, Santiago del Estero 2829, (3000) Santa Fe, Argentina.  
{\it E-mail: ridatole@gmail.com}}

\begin{abstract}
Given a suitable extension $F'/F$ of algebraic function fields over a finite field $\ff_q$, we introduce the conorm code $\con_{F'/F}(\cod)$ defined over $F'$ which is constructed from an algebraic geometry code $\cod$ defined over $F$.  We study the parameters of $\con_{F'/F}(\cod)$ in terms of the parameters of $\cod$, the ramification behavior of the places used to define $\cod$ and the genus of $F$. 
In the case of unramified extensions of function fields we prove that $\con_{F'/F}(\cod)^\perp = \con_{F'/F}(\cod^\perp)$ when the degree of the extension is coprime to the characteristic of $\ff_q$. We also study the conorm of cyclic algebraic-geometry codes and we show that some repetition codes, Hermitian codes and all Reed-Solomon codes can be represented as conorm codes.
\end{abstract}

\maketitle

\section{Introduction}
Let $\ff_q$ be a finite field with $q$ elements. For a given trascendental element $x$ over $\ff_q$, the field of fractions of the ring $\ff_q[x]$ is denoted as $\ff_q(x)$ and it is called a rational function field over $\ff_q$. An (algebraic) function field $F$  of one variable over $\ff_q$ is a field extension  $F/\ff_q(x)$ of finite degree where $x\in F$ is a trascendental element over $\ff_q$.  The finite field $\ff_q$ is called the field of constants of $F$ and  we  will always assume that $\ff_q$ is the full constant field of $F$, that is $\ff_q$ is algebraically closed in $F$. We will frequently use the symbol $F/\ff_q$ to express that $F$ is a function field over $\ff_q$.

Let $F/\ff_q$ be a function field. The \textit{Riemann-Roch space} associated to a divisor $G$ of $F$ is the vector space over $\ff_q$ defined as
$$\mathcal{L}(G)=\{x\in F\,:\, (x)\geq G\}\cup \{0\},$$ 
where $(x)$ denotes the principal divisor of $x$. It turns out that $\rr(G)$ is a finite dimensional vector space over $\ff_q$ for any divisor $G$ of $F$ (see, for instance, Proposition 1.4.9 of \cite{Sti}). 

The {\em dimension} $\ell(G)$ of a divisor $G$ of $F$ is defined as the dimension of $\mathcal{L}(G)$ as a vector space over $\ff_q$. 
An important result in the theory of algebraic function fields relates the dimension of some divisors and the genus of the considered function field. More precisely (see Theorem 1.5.15 of \cite{Sti}) given a function field $F/\ff_q$ of genus $g$, a divisor $G$ and a  canonical divisor $W$ of $F$, the Riemann-Roch Theorem asserts that
$$\ell(G)= \deg(G)+1-g+\ell(W-G).$$

Given disjoint divisors $D=P_1+\cdots+P_n$ and $G$ of $F/\ff_q$, where $P_1,\ldots,P_n$ are different rational places,   
the \textit{algebraic geometry code} (AG-code for short) associated to $D$ and $G$ is defined as
\begin{equation} \label{CDG}
C_\rr^F (D,G) = \{(x(P_1),\ldots, x(P_n))\,:\,x\in \rr(G)\}\subseteq (\ff_q)^n,
\end{equation}
where $x(P_i)$ denotes the residue class of $x$ modulo $P_i$ for $i=1,\ldots,n$. If the context is clear, we will simply write $C_{\rr}$ instead of $C_\rr^F (D,G)$.

It is well known (see Theorem 2.2.2 of \cite{Sti}) that $C_\rr(D,G)$ is an $[n,k,d]$-code with $k=\ell(G)-\ell(G-D)$ and 
\begin{equation} \label{d>}
d\geq n-\deg G.
\end{equation}
 Also, if $\deg G <n$, then 
\begin{equation} \label{k>}
k=\ell(G) \ge  \deg G+1-g 
\end{equation}
and hence $k+d \ge n+1-g$.
If, in addition $2g-2 < \deg G$, we have the equality  
$k= \deg G+1-g$. 

The \textit{designed distance} of the code is $d^*=n-\deg G$ 
and, similarly, one can define its \textit{designed dimension} as
\begin{equation}\label{desdim}
k^*=\deg G+1-g.
\end{equation}
Thus $d\ge d^*$ and if $\deg G<n$ then $k \ge k^*$, with equality if $2g-2<\deg G$.

From these facts, we see that if $\deg G>2g-2$, then the dimension of $C_\rr(D,G)$ can be computed in an exact way knowing the degree of $G$ and the genus of $F$. On the other hand, if $\deg G\leq 2g-2$, then $\ell(D-G)$ does not vanish and no formula is available to compute the dimension of $C_\rr(D,G)$.

Sometimes it is useful to distinguish 3 levels of AG-codes. Let $\cod=C_\rr(D,G)$ be as in \eqref{CDG}. If $\deg G<n$ we will say that $\cod$ is a \textit{moderate} AG-code (or \textit{MAG-code}). A moderate AG-code which also satisfies $2g-2<\deg G$ will be called a \textit{strong} AG-code (or \textit{SAG-code}). Finally, a \textit{weak} AG-code (or \textit{WAG-code}) will be an AG-code which is not moderate.

The main goals of this paper are to introduce the concept of the conorm code associated to an AG-code, to study some interesting properties of this new code and to show that some well known families of codes such as repetition codes, Hermitian codes and Reed-Solomon codes can be obtained as conorm codes from other more basic codes.

\subsubsection*{Outline and results}
In Section 2, given a suitable extension $F'/F$ of functions fields $F'/\ff_{q^t}$ and $F/\ff_q$, we define the $q^t$-ary conorm code $\cod' = C_\rr^{F'}(D',G')$ of the $q$-ary AG-code $\cod =C_\rr^F(D,G)$ by lifting the code $\cod$ using the conorm map of divisors. 
We denote this code as
$$\cod' = \con_{F'/F} (\cod).$$
If the context is clear, we will simply write $\con(\cod)$ instead of $\con_{F'/F}(\cod)$.
 
In the next section we deal with the parameters and levels of conorm codes. In Proposition~\ref{prop C'} and Corollary \ref{corito} we give bounds for the parameters of $\cod'$ in terms of the parameters of $\cod$. 
Then, we consider conorm codes defined over geometric extensions, that is extensions of function fields $F'/F$ having both the same field of constants $\ff_q$ (i.e.\@ $t=1$) and we study the AG-levels of the construction (see Corollary \ref{corito2}).

In Section 4, we study conorm codes defined over unramified extensions and duality. In general, the dual of the conorm code is not the conorm of the dual code (see Example \ref{counterexample}). However, over unramified extensions, this is indeed the case under some conditions. More precisely, in Theorem \ref{prop dual} we show that if $F'/F$ is an unramified geometric extension of degree $m$ of function fields over $\ff_q$ then 
$$\con(\cod^\perp) = \con(\cod)^\perp$$ holds 
provided that $(m,q)=1$.

\goodbreak 

In Section 5, we consider the conorm of cyclic AG-codes. We show that, under certain conditions this construction preserves cyclicity. In the particular case that $F'/F$ is a geometric Galois extension of function fields over $\ff_q$ and every place in the support of a divisor $D$ of $F$ is totally ramified in $F'$, the conorm code $\cod'=\con(C_\rr^F(D,G))$ defined over $F'$ and the AG-code $C_\rr^F(D,G)$ are different representations of the same algebraic geometry code over $\ff_q$ (see Theorem \ref{teo conorm}). 
We believe this may have some applications on code-based cryptography.

Finally, in the last section we show that in some general cases, repetition codes, Hermitian codes and Reed-Solomon codes can be represented as conorm codes, i.e.\@ they can be seen as the conorm code of simpler AG-codes defined over function fields of smaller genus.

\section{The conorm code of an AG-code}   \label{sec2}
Let $F/\ff_q$ be a function field and let us denote as usual the set of places of $F$ by $\pl(F)$ and the abelian group of divisors of $F$ by $\divisor(F)$. Let $F'/\ff_{q^t}$ be a function field such that $F'/F$ is a finite extension. We will show how to construct an AG-code $\cod'=C_\rr^{F'}(D',G')$ over $\ff_{q^t}$ starting from an AG-code 
$\cod=C_\rr^F(D,G)$ over $\ff_q$. This will be accomplished by using the conorm map on divisors 
$$\con_{F'/F} : \divisor(F) \rightarrow \divisor(F'),$$ 
that we now recall.  
If $P$ is a place in $F$, the conorm divisor of $P$ is the divisor 
$$\con_{F'/F}(P) = \sum_{_{P'|P}} e(P'|P) \, P'$$ 
in $F'$, where $e(P'|P)$ is the \textit{ramification index} of the place $P'$ in $F'$ over $P$. For $Q\in \pl(F)$ and  $A=\sum_P n_P P \in \divisor(F)$ we define $v_Q(A)=n_Q$. 
Now, the \textit{conorm divisor} of $A$ in $F'$ is given by
\begin{equation} \label{conorm}
A':= \con_{F'/F} (A) = \sum_P n_P \con_{F'/F}(P). 
\end{equation}

From now on the extension $F'/F$ is a function field extension of degree $m$.
Let $\cod=C_\rr^F(D,G)$ be an AG-code of length $n$ defined over $\ff_q$,
where $G$ and $D=P_1+\cdots+P_n$ are disjoint divisors and $P_1,\ldots,P_n$ are different rational places of $F$. 
For any place $P$ in the support of $D$ let us denote by $m_P \in \{1, \ldots, m\}$ the number of different places of $F'$ over $P$. 
Suppose that the extension $F'/F$ is such that for every place $P$ in the support of $D$  we have that 
\begin{equation}\label{condicion} 
e(P'|P) = \frac{m}{m_P},
\end{equation}
for any place $P'$ of $F'$ lying above $P$. Then  all the places of $F'$ lying above $P_i$ are rational for $i=1,\ldots, n$.
Denote by 
\begin{equation} \label{Qis}Q_i^{(1)},\ldots,Q_i^{(m_{P_{i}})}\end{equation}
the rational places of $F'$ lying above $P_i$ and put 
\begin{equation} \label{DrDs}
D_{i}=Q_i^{(1)}+\cdots+Q_i^{(m_{P_{i}})} .
\end{equation}

Note that by \eqref{conorm} we have
\begin{equation*} 
 \con_{F'/F} (P_i) = \tfrac{m}{m_{P_i}}\sum_{j=1}^{m_{P_i}} Q_i^{(j)} = \tfrac{m}{m_{P_i}} D_i 
\end{equation*}
and since $D$ and $G$ are disjoint divisors of $F$ then $\con_{F'/F}(D)$ and $\con_{F'/F}(G)$ are disjoint divisors of $F'$.

With the above notation we have the following definition of an AG-code ``hanging over'' another one.
\begin{defi} \label{concod}
Given a code $\cod = C_\rr^F(D,G)$ as in \eqref{CDG} and a finite extension $F'/F$ of function fields such that \eqref{condicion} holds, we define the conorm code associated to $\cod$, or just the \textit{conorm of $\cod$}, as 
\begin{equation} \label{Con C}
\cod' = \con_{F'/F} (\cod) = C_\rr^{F'}(D',G'), 
\end{equation}
where
$$D' =  \tfrac{1}{m} \sum_{i=1}^n m_{P_i}\con_{F'/F} (P_i) 
\qquad \text{ and } \qquad G'=\con_{F'/F} (G).$$
That is, in the notation of \eqref{Qis} and \eqref{DrDs},
$$D' =D_1+\cdots+D_n=  \sum_{i=1}^{n} (Q_i^{(1)} + \cdots + Q_i^{(m_{P_i})}).$$
When $F'/F$ is understood, we will write $\con(\cod)$ instead of $\con_{F'/F} (\cod)$. Similarly for 
$\con_{F'/F} (P_i)$ and $\con_{F'/F} (G)$.
\end{defi}

Clearly $\cod'$ is an AG-code defined over $F'$. For $m=1$ the construction is trivial and $\cod'=\cod$. 
By Hurwitz genus formula (see Theorem 3.4.13 of \cite{Sti}), the genus $g'=g(\podesta{F'})$ of $F'$ is given by 
\begin{equation} \label{g'}
g' = \tfrac{m}{t} (g-1) + \tfrac 12 \deg \mathrm{Diff}(F'/F)  + 1 ,
\end{equation}
where $g=g(F)$ is the genus of $F$ and
$$\mathrm{Diff}(F'/F) = \sum_{P} \sum_{P'|P} d(P'|P) \, P',$$
is the \textit{different divisor} of $F'/F$ with $d(P'|P)$ the \textit{different exponent} of $P'$ over $P$. Hence, since $d(P'|P)\ge 0$ for every $P'|P$, 
we have
\begin{equation}\label{g' >= g}
g'=g(F') \ge g(F)=g.
\end{equation}

\section{Parameters and levels}
We  study now some parameters and levels of the conorm code $\con(\cod)$ of an AG-code $\cod$ in different situations. 
We begin with the following elementary estimates for the parameters of a conorm code.

\begin{prop} \label{prop C'}
Let $F'/\ff_{q^t}$ and $F/\ff_q$ be two function fields such that $F'/F$ is a finite extension  of degree $m\ge 2$. 
Let $[n,k,d]$ and  $[n',k',d']$ be the parameters of $\cod=C_\rr(D,G)$ and $\cod'=\con_{F'/F}(\cod)$ respectively, where $D=P_1+\cdots+P_n$.
Then
\begin{equation} \label{n'}
n \le n' \le  mn,
\end{equation}
and
\begin{equation} \label{d'}
d' \ge n'-\tfrac{m}{t} \deg G.
\end{equation}
Moreover if $\cod'$ is a MAG-code (i.e.\@ $\deg G' < n'$) then
\begin{equation} \label{k'}
k' \ge \tfrac{m}{t} k^* - \tfrac 12 \deg \mathrm{Diff} (F'/F),
\end{equation}
where $k^*=\deg G+1-g$ is the designed dimension given in \eqref{desdim}. 
\end{prop}

\begin{proof}
We see at once that \eqref{n'} holds because $n' = \# \sop (D')$ and, by definition of the conorm code, we have 
$$n' = \sum_{i=1}^n m_{P_i},$$ 
where $1\leq m_{P_i}\leq m$ for $i=1,\ldots,n$.  

We prove now the lower bounds \eqref{d'} and \eqref{k'}. 
By \eqref{d>} we have that  
$$d' \ge n'-\deg G'.$$ 
We see that \eqref{d'} holds because from Corollary 3.1.14 in \cite{Sti} we have 
\begin{equation} \label{degG'G}
\deg G' = \tfrac{m}{t}\deg G.
\end{equation}
Finally from \eqref{k>} we have that  if $\deg G' < n'$
then $k'\ge \deg G' +1-g'$. From this, and using \eqref{g'} and \eqref{degG'G}, we see that \eqref{k'} also holds. 
\end{proof}

With the same hypothesis of Proposition \ref{prop C'} we have the following
\begin{coro} \label{corito} 
Let $s$ (resp. $r$) be the number of places $P_i$ in $D=P_1+\cdots+P_n$ which split completely (resp. are totally ramified) in $F'$, and assume that $n=r+s$. Then
\begin{enumerate}[($a$)]
	\item the length $n'$ of the conorm code $\cod'=\con_{F'/F}(\cod)$ satisfy \begin{equation}\label{n'bis}
	n+s\le n'=ms+r \le mn-r,
	\end{equation} and equalities hold if and only if the extension $F'/F$ is quadratic ($m=2$). \sk 
	
	\item $n'=mn$ if and only if $s=n$ and $r=0$; and in this case, $d' \ge m(n-\tfrac{\deg G}{t})$. \sk 
	
	\item $n'=n$ if and only if $s=0$ and $r=n$; and, in this case, $d' \ge n-\tfrac{m}{t} \deg G = n-\deg G'$.
\end{enumerate}
\end{coro}

\begin{proof}
($a$) It is straightforward to check that both inequalities in \eqref{n'} hold if and only if $m=2$. In this case, $n+s=2s+r=2n-r$. 

($b$) Since $r=n-s$, we have that $n'=mn$ if and only if
 $(m-1)s=(m-1)n$, which holds if and only if $s=n$ (and hence $r=0$), since $m>1$. 
 
($c$) Similarly, $n'=n$ if and only if $(m-1)s=0$, which in turn can only happen if $s=0$ since $m>1$. The assertions on the distance are clear now from \eqref{d'}.
\end{proof}

\subsubsection*{Quadratic extensions}
Suppose $F'/\ff_{q^t}$ is a quadratic extension of $F/\ff_q$. Since $t\mid m$ and $m=2$, then $t=1$ or $t=2$. If $t=1$ then $F'/F$ is a geometric extension. This case will be studied in the next paragraph. Thus, assume that $t=2$. 
In this case we have that $F'$ is a constant field extension of $F$ so that $F'/F$ is an unramified  extension (\cite[Thm. 3.6.3]{Sti}). Then
$$n'=2s=2n.$$ 
We know that $d'\ge n'-\deg G'$ and $d\ge n-\deg G$. Thus the bound for $d'$ can be improved since by \eqref{d'} we have 
$$d' \ge 2s-\deg G = n'-\deg G,$$
or in other terms 
$$d' \ge (n-\deg G) +s.$$
Regarding the dimension, 
we have $\deg(\mathrm{Diff}(F'/F))=0$, since the extension is not ramified, and hence we get 
$$k'\ge k^*.$$
So, in general, for conorm codes over non-geometric quadratic extensions the minimum distance and the dimension may increase.

\subsubsection*{Geometric extensions and levels}
We  consider now the particular case of geometric extensions, that is finite extensions $F'/F$ of algebraic function fields over the same field of constants $\ff_q$. Notice that in this case, the bounds for the parameters in Proposition \ref{prop C'} and Corollary \ref{corito} hold with $t=1$.
Recall that the secondary parameters of an $[n,k,d]$-code are the information rate $R=k/n$ and the relative minimum distance 
$\delta = d/n$.

\begin{coro} \label{corito2}
Let $F'/F$ be a geometric extension of function fields over $\ff_q$ of degree $m>1$. Let $\cod'= \con(\cod)$ as in Corollary \ref{corito} with $n=r+s$. The following holds:
\begin{enumerate}[$(a)$]
\item If $\cod'$ is a MAG-code, then $\cod$ is a MAG-code.  
If $r=0$, then the converse also holds and $d'\ge 2$. If further 
$d=n-\deg G$, then $d'\ge md$ and $\delta'\ge \delta$. \msk

\item If either $\cod'$ is a MAG-code and $2g-2<\deg G$, or else $\cod$ is a SAG-code and $r=0$, then 
$k' \ge mk - \tfrac 12 \deg \mathrm{Diff} (F'/F)$.
\end{enumerate}
\end{coro}

\begin{proof}
$(a)$ Since $\cod'$ is a MAG-code, 
$$\deg G'<n'=ms+r.$$ 
Also, $\deg G' = m\deg G$ and $n=s+r$, thus 
$$\deg G<s+\tfrac rm<n.$$ 
If $r=0$, then $n'=nm$ and hence $\deg G<n$ implies that $\deg G' <n'$. In this case, $d'\ge m(n-\deg G) >m\ge 2$.
If in addition $d=n-\deg G$ then $d'\ge md$ and hence 
$$\delta' = \frac{d'}{n'} \ge\frac{md}{mn} = \frac{d}{n} = \delta.$$

$(b)$ Since $\cod'$ is a MAG-code we have \eqref{k'}. By $(b)$, $\cod$ is also a MAG-code and since $2g-2<\deg G$ by assumption, $\cod$ is a SAG-code. This implies $k=\deg G+1-g$ and hence, by \eqref{k'}, we have
$$k'\ge mk-\tfrac 12 \deg \mathrm{diff}(F'/F).$$ 
Now, if $\cod$ is a SAG-code (in particular a MAG-code), then the hypothesis $r=0$ implies, by $(a)$, that $\cod'$ is a MAG-code, and hence we are in the previous case.
\end{proof}

\subsubsection*{Examples in quadratic geometric extensions}
We now give examples of conorm codes in quadratic geometric extensions of some rational function field. 

\begin{exam}
Consider $F=\ff_4(x)$ the rational function field over $\ff_4=\{0,1,\alpha, \alpha^2\}$, where $\alpha^2+\alpha+1=0$. We define the SAG-code $\cod=C_\rr(D,G)$ with $$D=P_1+P_\alpha+P_{\alpha^2}\qquad \text{and}\qquad G=2P_\infty$$ where $P_1$, $P_\alpha$ and $P_{\alpha^2}$ are the rational places which are simple zeroes of $x+1$, $x+\alpha$ and $x+\alpha^2$ respectively, and $P_\infty$ is the rational place that is a simple pole of $x$ in $F$. 
	We have that $\cod$ is a SAG-code over $\ff_4$ with parameters $[3,3,1]$, hence MDS (maximum distance separable). In fact, since $\deg G=2$ and $g(F)=0$, we have that $\cod$ is a SAG-code. Also, $k=2+1-0=3$ and $d\geq 3-2=1$ but, by Singleton bound, we know that $d\leq 1$. 
	
	Let us now consider $F'=F(y)=\ff_4(x,y)$ where 
	$$y^2+y=\frac{x^2}{x+1}.$$ 
	This extension $F'/F$ is the first step of a famous tower of function fields given by Garcia and Stichtenoth 
	in \cite{GS96}. Since $F'/F$ is an Artin-Schreier extension, we have that $P_1$ and $P_\infty$ are totally ramified in $F'$ while $P_\alpha$ and $P_{\alpha^2}$ split completely in $F'$. 
	Moreover, we have 
	$$[F':F]=2 \qquad \text{and} \qquad g(F')=1.$$ 
	In this case $\cod'=\con_{F'/F}(\cod)=C_\rr(D',G')$ is also a SAG-code over $\ff_4$ with $$D'=Q_1+R_\alpha+S_\alpha+R_{\alpha^2}+S_{\alpha^2} \qquad \text{and}\qquad G'=4Q_\infty,$$ where $Q_1$ (resp. $Q_\infty$) is the only place over $P_1$ (resp. $P_\infty$) and $R_\alpha$ and $S_{\alpha}$ (resp. $R_{\alpha^2}$ and $S_{\alpha^2}$) are the two places over $P_\alpha$ (resp. $P_{\alpha^2}$). We have now that $n'=5$, $k'=4$ and $d'\geq 1$. Thus, $\cod'$ is a $[5,4,d']$-code over $\ff_4$ with $1\leq d'\leq 2$. 
	In fact $d'=1$ because if $z=(x-\alpha)(x-\alpha^2)$, the principal divisor $(z)^{F'}$ of $z$ in $F'$ is
	\[(z)^{F'}=R_{\alpha}+S_{\alpha}+R_{\alpha^2}+S_{\alpha^2}-4Q_{\infty}.\]
This implies that $z\in \rr(G')$ and also that
\[z(Q_1)\neq 0\quad\text{and}\quad z(R_{\alpha})=z(S_{\alpha}) =z(R_{\alpha^2})=z(S_{\alpha^2})=0.\]
Thus there is a codeword in $\cod'$ of weight $1$.
\end{exam}

\begin{exam}
Let $F=\ff_q(x)$ be a rational function field and consider the quadratic extension $F'/\ff_q$ of $F/\ff_q$ determined by the elliptic function field $F'=\ff_q(x,y)$
given by 
$$y^2 =f(x)$$ 
where $f(x) \in \ff_q[x]$ is square-free of degree 3.  
We will fix a rational AG-code $\cod=C_\rr^F(D,G)$ and we will consider the elliptic conorm code $\cod' = \con_{F'/F}(\cod)$. 

Let $R_1, \ldots, R_q$ and $P_\infty$ be the rational places of $F$ and let $P_1, \ldots, P_r$ be the places corresponding to the irreducible monic polynomials $p_i(x)$ in the factorization of $f(x)$, hence $1\leq r \leq 3$, and put 
$$D_1=R_1+\cdots +R_q \quad \text{and} \quad D_2=P_1+\cdots+P_r.$$
There are various possibilities for $\cod$, let us see three of them. \sk

\noindent ($i$) 
Assume that $\ff_q$ has odd characteristic and consider
$$\cod_1 = C(D_1, \ell P_\infty)$$ 
for $\ell \in \na$. Then, $P_1,\ldots,P_r$ and $P_\infty$ are the only ramified places of $F'$. Even more, the places $P_1,\ldots,P_r$ and $P_\infty$ are totally ramified in $F'$. If $Q_1,\ldots,Q_r$ and $Q_\infty$ denote the corresponding places of $F'$ over them, then $\deg Q_j=\deg P_j$ and $\deg Q_\infty =1$.
Thus, the conorm codes $\con(\cod_1)$ has parameters $[n_1',k_1',d_1']$ where 
$$n_1'=\left\{
\begin{array}{cl}
  2q  & \qquad \text{if $f$ is irreducible, } (r=1),     \\[1.5mm]
  2q-1& \qquad \text{if $f$ has only one linear factor, } (r=2),     \\[1.5mm]
  2q-3& \qquad \text{if $f$ has three linear factors, } (r=3). 
\end{array}
\right.$$
By \eqref{d'} and \eqref{k'} in Proposition \ref{prop C'}, we have  
$$d_1' \ge n_1'-2\ell$$ and, 
since $\mathrm{Diff} (F'/F) = Q_1+\cdots +Q_r+Q_\infty$ we have that $\deg \mathrm{Diff} (F'/F) =4$ and hence,  
$$k_1' \ge 2(k_1^*-1)=2\deg G =2\ell.$$ 
Thus, by the above expressions and the Singleton bound, we have that 
$$n_1' \le d_1'+k_1' \le n_1'+1.$$
This is in coincidence with the known fact that elliptic codes are almost MDS, that is they are MDS, or the Singleton bound fails by one.  \sk

\noindent
($ii$) Another possibility is to take 
$$\cod_2 = C(D_1+P_\infty, D_2) \qquad \text{or} \qquad 
\cod_3 = C(D_1, D_2 + \ell P_\infty)$$
for $\ell \ge 1$, with parameters $[q+1,k_2,d_2]$ and $[q,k_3,d_3]$, respectively. Here, to ensure that the supports of the divisors $D$ and $G$ are disjoint we have to assume  that $f$ is irreducible over $\ff_q$.  
The associated elliptic conorm codes $\cod_2'=\con_{F'/F}(\cod_2)$ and $\cod_3'=\con_{F'/F}(\cod_3)$
have parameters $[2q+1,k_2',d_2']$ and $[2q,k_3',d_3']$, respectively.
Similarly as in ($i$), one can obtain bounds for the dimension and minimum distance of these codes.
\end{exam}

\section{Unramified extensions and duality}
We consider now unramified extensions $F'/F$ of function fields over $\ff_q$. Under this assumption, it is possible to get some nice results for duality of conorm codes.

We begin by studying the relation between the AG-levels of an AG-code and its conorm code.

\begin{prop} \label{coro unram}
Let $F'/F$ be an unramified extension of algebraic function fields over $\ff_q$ of degree $m$. Then $\cod'$ is a SAG-code (resp.\@ MAG) if and only if $\cod$ is a SAG-code (resp.\@ MAG). In this case, $k'=mk$ and $R'=R$. If in addition $d=n-\deg G$, 
then $d'\ge md$ and $\delta'\ge \delta$.\end{prop}

\begin{proof}
Recall that $F'/F$ is unramified if and only if $\mathrm{Diff} (F'/F) =0$. Hence, 
$$g'-1=m(g-1), $$ 
by \eqref{g'}.
Also, $r=0$ and $n'=mn$. By $(b)$ of Corollary \ref{corito}, $\cod'$ is a MAG-code if and only if $\cod$ is a MAG-code.
Since 
$$2m(g-1)=2(g'-1)<\deg G'=m\deg G ,$$ 
we see that $\cod'$ is a SAG-code if and only if $\cod$ is a SAG-code.
In this situation, we have both $k'=\deg G'+1-g'$ and $k=\deg G + 1-g$. Putting together all these information we have
$$k'= m \deg G+ m(1-g) = 
m(\deg G+1-g)=mk,$$ 
and then $R' = k' / n' = mk / mn = k/n= R$.

For the remaining assertion, by \eqref{d'} we have
$$d' \ge m(n-\tfrac{\deg G}{t}) \ge m(n-\deg G) = md,$$
from which we have $\delta' = d'/n' \ge md/mn = \delta$, as we wanted to show.
\end{proof}

\begin{exam}
Let $F_0=K(x_0)$ be the rational function field over $K=\ff_{4^3}$ and consider the finite tower 
$$F_0\subset F_1 \subset F_2 \subset F_3$$
of functions fields over $K$, where each field extension is a Kummer extension of degree 3 recursively defined for $i=1,2,3$ by $F_i=F_{i-1}(x_i)$ where 
$$ x_i^3 = 1 + \frac{x_{i-1}^3}{(x_{i-1}-1)^3}.$$ 

Let us denote by $P_\beta$ the simple zero of {\blue $x_0-\beta$, for $\beta \in \ff_4=\{0,1,\alpha,\alpha^2\}$}, and by $P_\infty$ the simple pole of $x_0$ in $F_0$. In \cite{Wu04}, Wulftange proved that $P_0$ splits completely in $F_3/F_0$, $P_1$ splits completely in $F_1/F_0$, is totally ramified in $F_2/F_1$ and then splits completely again in $F_3/F_2$. Furthermore $P_{\alpha}$, $P_{\alpha^2}$ and $P_\infty$ are totally ramified in $F_1/F_0$ and then they split completely in $F_3/F_1$. The ramification behavior of the other rational places of $F_0$ was studied in \cite{Wu04}, where the author also proved that the extension $F_2/F_1$ is ramified but the extension  $F_3/F_2$ is unramified.

Using the above description of the the ramification behavior and the Hurwitz genus formula, we have that $g(F_2)=4$. 
We also have that  there are exactly nine places $Q_1,\ldots,Q_9$ of $F_2$ lying above $P_0$, three places $Q_{10}, Q_{11}$ and $Q_{12}$ of $F_2$ lying above $P_1$  and three  places $R_1, R_2$ and $R_3$ lying above $P_\infty$. All of them are rational places of $F_2$. We define $\cod=C_\rr(D,G)$ with 
\[D=Q_1+\cdots+Q_{12} \quad \text{and} \quad G= 3R_1+3R_2+3R_3.\] 
Since
\[2g(F_2)-2=6< 9=\deg G < 12=n,\]
we have that $\cod$ is a $[12,6,d]$ SAG-code with $d\geq 3$. In fact $d=3$ because the principal divisor $(x_0)^{F_2}$ of $x_0$ in $F_2$ is
\[(x_0)^{F_2}=Q_1+\cdots+Q_9-G,\]
so that $x_0\in \rr(G)$ and also
\[x_0(Q_1)=\cdots=x_0(Q_9)=0\quad\text{and}\quad x_0(Q_{10})\neq 0, x_0(Q_{11})\neq 0, x_0(Q_{12})\neq 0.\]
This implies that there is a codeword of $\cod$ of weight $3$. Now, since $F_3/F_2$ is unramified, we see from Proposition \ref{coro unram} that the conorm code $\cod'$ of $\cod$ is also a SAG-code with $n'=36$, $k'=18$ and $d'\geq 9$. In fact $d'=9$  because the principal divisor $(x_0)^{F_3}$ of $x_0$ in $F_3$ is
\[(x_0)^{F_3}=Q'_1+\cdots+Q'_{27}+Q'_{28}+\cdots+Q'_{36}-G',\]
where $Q'_1,\cdots,Q'_{27}$ are all the places of $F_3$ lying over $P_0$,  $Q'_{28},\cdots,Q'_{36}$ are all the places of $F_3$ lying over $P_1$ and $G'=\con_{F_3/F_2}G$. Thus $x_0\in \rr(G')$ and we also have that
\[x_0(Q'_1)=\cdots=x_0(Q'_{36})=0\quad\text{and}\quad x_0(S)\neq 0,\]
for any $S\in \{Q'_{28},\ldots, Q'_{36}\}$. This implies that there is a codeword of $\cod'$ of weight $9$. In particular we see that the lower bound for $d'$ given in Proposition \ref{coro unram} can not be improved in general.
\end{exam}

\subsubsection*{Duality}
In general, the dual of the conorm code is not the conorm of the dual code. However, this is indeed the case for conorm codes defined over unramified extensions with an additional condition. More precisely

\begin{thm} \label{prop dual}
Let $F'/F$ be an unramified finite extension of algebraic function fields of degree $m$ over $\ff_q$ such that $\gcd(m,q)=1$ and let $\cod=C_\rr(D,G)$. Then  
\begin{equation} \label{con dual}
\con(\cod^\perp) = \con(\cod)^\perp.
\end{equation}
\end{thm}

\begin{proof}
Let $\cod=C_\rr(D,G)$ and $\con(\cod)=C_\rr(D',G')$ with $D'=\con(D)$ and $G'=\con(G)$. On the one hand, from Definition 2.2.6 and Theorem 2.2.8 in \cite{Sti}, we have that $\con(\cod)^\perp=C_{\Omega}(D',G')$. 

On the other hand, we know that $\cod^\perp=C_{\Omega}(D,G)$ and by Lemma 2.2.9 and Proposition 2.2.10 in \cite{Sti} there exist a Weil differential $\eta$ of $F$ such that 
$$\cod^\perp=C_{\Omega}(D,G)=C_\rr(D,H) \quad \text{with} \quad H=D-G+(\eta)$$ 
and also $v_{P_i}(\eta)=-1$ and $\eta_{P_i}(1)=1$ for all $i=1,\ldots,n$ where $\{P_1, \ldots, P_n\} = \sop(D)$. 

Then $$\con(\cod^\perp)=C_\rr(D', H') \quad \text{with} \quad H'=\con(H)=D'-G'+\con((\eta)).$$

Let  $\eta'=\cot_{F'/F}(\eta)$ be the cotrace of $\eta$, that is $\eta'$ is a Weil differential of $F'$ such that 
(see \cite{Sti}, Theorem 3.4.6) 
$$(\eta')=(\cot_{F'/F}(\eta))=\con_{F'/F}((\eta))+\mathrm{Diff}(F'/F),$$
and since in this case the extension is unramified we have 
$$(\eta')=\con((\eta)).$$ 

Moreover, if $\sop(D')=\{Q_1,\ldots, Q_{n'}\}$ then for each $j$ there is an index $i$ such that $$Q_j \cap F=P_i\in \sop(D),$$ and since $P_i \not\in\sop(H)$, because $v_{P_i}(\eta)=-1$, then $Q_j \not\in \sop(H')$. Thus, we have  
\begin{eqnarray*}
0 &=& v_{Q_j}(H') = v_{Q_j}(D'-G'+(\eta')) \\ 
&=& v_{Q_j}(D')-v_{Q_j}(G')+v_{Q_j}(\eta')=1-0+v_{Q_j}(\eta').
\end{eqnarray*}

Since $m=[F':F]$ is coprime with $q$, we can consider $\bar{m}\in \ff_q^{\ast}$ and its inverse ${\bar{m}}^{-1}$ in the multiplicative group $\ff_q^\ast$ and define $\tilde{\eta}=\bar{m}^{-1}\,\eta'$. 
Therefore, $v_{Q_j}(\tilde{\eta})=v_{Q_j}(\eta')=-1$ for each $j=1, \ldots, n'$.
Moreover, we also have  
$$\tilde{\eta}(1) = \bar{m}^{-1} \eta'(1) = \eta' (\bar{m}^{-1}) = \tr_{\ff_q/\ff_q}(\eta' (\bar{m}^{-1})) = \eta(\tr_{F'/F} (\bar{m}^{-1})) = \eta(1)=1.$$
Then, by using Proposition 2.2.10 in \cite{Sti} again, we have 
$$C_{\Omega}(D',G')=C_\rr(D',D'-G'+(\tilde{\eta})).$$
Finally, putting all these things together we have that 
\begin{eqnarray*}
\con(\cod^\perp) & = &  C_\rr(D',H') = C_\rr(D',D'-G'+(\tilde{\eta})) \\ 
&=& C_{\Omega}(D',G') = (C_\rr(D',G'))^\perp  =(\con(\cod))^\perp
\end{eqnarray*}
as we wanted to show.
\end{proof}

\begin{rem}
Theorem \ref{prop dual} probably holds not only for unramified extensions. In general we have that 
$$\con(\cod)^\perp = C(D',D'-G'+(\eta')) \quad \text{and} \quad \con(\cod^\perp) = C(D',D'-G'+ \con(\eta)),$$ 
where $\eta$ (resp. $\eta'$) is a Weil differential of $F$ (resp. $F'$). 
Thus, by the results of Munuera and Pellikaan in \cite{MP}, the problem in proving that $\con(\cod)^\perp = \con(\cod^\perp)$ reduces to determine whether or not the canonical divisors $(\eta')$ and $\con(\eta)$ are equal or rational equivalent.
\end{rem}

\section{The conorm of cyclic AG-codes} \label{sec3}
Here we assume, as we did in Section 3, that all the extensions of function fields considered are geometric. By using Galois extensions we 
study the conorm codes of cyclic AG-codes.
We will show that under certain conditions on the ramification behavior of the rational places of $D$, 
we can represent a cyclic AG-code $\cod=C_\rr^F(D,G)$ defined over $F$ as a cyclic AG-code defined over $F'$ by using the conorm. 

\sk 
First we will need some auxiliary results. {Let $F'/F$ be a function field extension and let $G'\in \divisor(F')$. The set of all places of $F$ lying below the places in $\sop(G')$ will be denoted as $\sop(G') \cap F$. In other words
$$\sop(G') \cap F=\{Q'\cap F: Q' \in \sop(G')\}.$$}  
 
\begin{lem} \label{lema1conorm} 
Let $F'/F$ be a extension of algebraic function fields over $\ff_q$.
Assume that $G \in\divisor(F)$ and $G' \in \divisor(F')$ are such that
$\sop(G') \cap F = \sop(G)$. If $v_{Q'}(G') \geq e(Q'|Q)\, v_Q(G)$ for every $Q' \in \sop(G')$ and $Q=Q' \cap F$, then  
$$\rr(G) \subseteq \rr(G').$$
\end{lem}

\begin{proof}
Let $Q' \in \sop(G')$, put $Q=Q' \cap F$ and take $x \in \rr(G)$. By hypothesis $Q \in \sop(G)$ and therefore $v_Q(x) \geq -v_Q(G)$. Then we have, 
\begin{eqnarray*}
v_{Q'}(x) &=& e(Q'|Q) \, v_Q(x) \geq  - e(Q'|Q) \, v_Q(G) \\
&\geq & - e(Q'|Q) \big( \tfrac{1}{e(Q'|Q)} v_{Q'}(G') \big) = -v_{Q'}(G'),
\end{eqnarray*}
and thus $x \in \rr(G')$.
\end{proof}

\begin{rem}
The result in the previous lemma holds in a more general situation, namely when $\sop(G') \cap F \subseteq \sop(G)$, provided that $G=G_0-G^+$ where $\sop (G_0) = \sop (G') \cap F$ and $G^+ \in \divisor(F)^+$ is a positive divisor of $F$. In the particular case that $G^+=0$ we are in the situation of Lemma  
\ref{lema1conorm}.
\end{rem}

\begin{lem}\label{lema2conorm}
Let $F'/F$ be a finite extension of algebraic function fields over $\ff_q$ of degree $m=[F':F]$.
If $G\in \divisor(F)$ and $G'=\con_{F'/F} (G)$ then $\rr(G) \subseteq \rr(G')$. Furthermore, 
\begin{enumerate}[$(a)$]
\item $\rr(G)\subseteq \tr(\rr(G'))$ if $(m,q)=1$, and \msk 

\item\label{itemb} $\tr(\rr(G')) \subseteq \rr(G)$ if $F'/F$ is Galois,
\end{enumerate}
where $\tr(\rr(G'))=\{\tr(x): x \in \rr(G')\}$  and $\tr$ is the trace map from $F'$ to $F$.
\end{lem}

\begin{proof}
The fact that $\rr(G) \subseteq \rr(G')$ follows from the previous Lemma.

\sk ($a$)
 Now, let $x \in \rr(G)\subseteq F$. 
Note that in this case $\tr(x)=mx$. Let $x'=m^{-1} x$ where $m^{-1}$ is the inverse of $m$ modulo $p=char(\ff_q)$. Then, if $Q'\in \pl(F')$, we have that $\tr(x')=m^{-1}\tr(x)=x$ and 
$$v_{Q'}(x') = v_{Q'}(m^{-1}x)=v_{Q'}(x) =  e(Q'|Q) \, v_Q(x) \geq -e(Q'|Q) \, v_Q(G) = -v_{Q'}(G').$$
Thus, $x' \in \rr(G')$ and therefore $x \in  \tr(\rr(G'))$.

\sk

($b$) Now let us assume that $F'/F$ is Galois.  In this case we have that $e(Q'|Q)=e_Q$ is the same for every $Q'|Q$ and 
if $\G=Gal(F'/F)$ then $\tr(x)=\sum_{\sigma \in \G} \sigma(x)$.
Let $x\in \tr(\rr(G'))$, i.e.\@ $x=\tr(x')$ with $x'\in \rr(G')$. Then $v_{Q'}(x')\geq -v_{Q'}(G') $ for every $Q'\in \pl(F')$.
Let $Q \in \sop(G)$ and $Q'|Q$. 
Note that 
\begin{eqnarray*}
v_{Q'}(x) &=& v_{Q'}(\tr(x')) = v_{Q'} \Big( \sum_{\sigma \in \G} \sigma(x') \Big) \ge \min_{\sigma \in \G}\{v_{Q'}(\sigma(x')\} \\
&=& v_{Q'}(\sigma_0(x')) = v_{\sigma_0^{-1}(Q')}(x') = v_{Q''}(x') \geq -v_{Q''}(G'),
\end{eqnarray*} 
for some $\sigma_0$ and some $Q''|Q$. Then, by the above calculation, we have
$$v_Q(x) = \tfrac{1}{e_Q}v_{Q''}(x) 
=\tfrac{1}{e_Q}v_{Q'}(x) 
\geq \tfrac{-1}{e_Q} v_{Q''}(G') = -v_Q(G),$$ 
and thus $x \in \rr(G)$. 
\end{proof}

\begin{rem} \label{obs norma}
Item ($b$) also holds if we replace the trace map with the norm map 
from $F'$ to $F$, since $N(x')=\prod_{\sigma \in \G} \sigma(x')$ and we have 
$$v_{Q'}(N(x')) = v_{Q'} \Big( \prod_{\sigma \in \G} \sigma(x') \Big) = 
\sum_{\sigma \in \G}v_{Q'}(\sigma(x')) \ge \min_{\sigma \in \G} \{v_{Q'}(\sigma(x'))\}.$$
\end{rem}

The next result is a direct consequence of Lemmas \ref{lema1conorm} and \ref{lema2conorm}.
\begin{coro}\label{coro1conorm} 
If $F'/F$ is a finite Galois extension of degree $m$ and $G'=\con_{F'/F} (G)$, 
then $\tr(\rr(G')) \subseteq \rr(G)\subseteq \rr(G')$.
If, in addition, $(m,q)=1$, then we have 
\begin{equation} \label{trLG}
\tr(\rr(G'))= \rr(G).
\end{equation}
\end{coro}

\subsubsection*{Totally ramified places and cyclicity of AG-codes} 
We recall that a linear code $C \subseteq \ff_q^n$ is cyclic if it is closed under the cyclic shift of its coordinates. Namely, 
for every $(c_1,\ldots,c_n)\in C$ we have that the word $(c_n,c_1,\ldots,c_{n-1})$ is also in $C$.   

We will show that the conorm construction preserves cyclicity in a special case. 
First we give this simple result.

\begin{lem} \label{lemin subcod}
Let $F'/F$ a finite extension of function fields over $\ff_q$. Let $\cod=C_\rr^F(D,G)$ be an AG-code such that $\sop (D)$ has only totally ramified places.
Then $\cod$ is a subcode of its conorm code $\cod'=\con (\cod)$ and $\cod$ is cyclic if $\cod'$ is cyclic.
\end{lem}

\begin{proof}
Let $D=P_1+\cdots+P_n$ with every $P_i$ totally ramified in $F'$. Let $Q_i$ be the only place of $F'$ over $P_i$, for $i=1,\ldots,n$, and thus $D'=\sum_{i=1}^n Q_i$. 
Now, consider $c=(c_1,\ldots,c_n)\in \cod$. Then, there is some $x \in \rr(G)$ such that $c_i=x(P_i)$ for every $i=1,\ldots,n$. Since $G'=\con(G)$, we have that $\rr(G) \subset \rr(G')$, 
by Lemma~\ref{lema1conorm}. 
Hence, $x(Q_i)=x(P_i)=c_i$ and therefore $c \in \cod'$.
The remaining assertion is clear.
\end{proof}

We recall the condition for an AG-code $\cod=C_\rr(D=P_1+\cdots+P_n,G)$ to be cyclic. 
Any codeword of $\cod$ is of the form $c=(x(P_1), x(P_2), \ldots, x(P_n))$ with $x\in \rr(G)$. The code 
$\cod$ is cyclic if and only if $s(c) = (x(P_n), x(P_1), \ldots, x(P_{n-1}))$ is also in $\cod$. 
But this happens if and only if there exists $z\in \rr(G)$ such that  
\begin{equation} \label{cyc cond}
(x(P_n),x(P_1), \ldots,x(P_{n-1})) = (z(P_1),z(P_2), \ldots, z(P_n)).
\end{equation}
That is, $\cod$ is cyclic if and only if for each $x\in \rr(G)$ there is a $z=z(x)\in \rr(G)$ such that \eqref{cyc cond} holds.

\begin{exam}
Let $F'/F$ be an algebraic function field extension over $\ff_q$ of finite degree $m$. 
Let $\cod=C_\rr^F(P,G)$ where $P\notin \sop(G)$ is a rational place which splits completely in $F'$. The code $\cod$, being of length 1, is just  $\{0\}$ or the whole $\ff_q$ depending on the degree of $G$, and hence trivially cyclic.
If $\cod$ is not trivial, its conorm code is $\cod'=C_\rr^{F'}(D',G')$, where $D'=\sum_{P'|P} P'$ and $G'=\con(G)$. 
If $F'/F$ is a cyclic extension, the elements in the Galois group 
$\mathcal{G}=\mathrm{Gal}(F'/F)$ permute the places over $P$ and hence it is (equivalent to) a cyclic code.
\end{exam}

We will need the following well-known equality for the residue classes of a rational place on $F$ and any place in $F'$ 
lying over it. 

\begin{lem} \label{lemitin}
	Let $F'/\ff_{q^t}$ be a finite extension of algebraic function fields of $F/\ff_q$.
	If $P$ is a rational place of $F$, then $x(P)=x(Q)$ for any $x$ in the valuation ring $\mathcal{O}_P$  of the place $P$ 
	and every place $Q$ of $F'$ lying over $P$.
\end{lem}

\begin{proof}
	Since $P$ is a rational place of $F/\ff_q$, the residue class field $F_P=\mathcal{O}_P/P$ of $P$ is just $\ff_q$, and thus $x(P)\in \ff_q$, for every 
	$x \in \mathcal{O}_P$. 
	This means that there is some $\alpha \in \ff_q$ such that $v_P(x-\alpha) > 0$. 
	That is, we have $x-\alpha \in P$ and hence 
	$0=(x-\alpha)(P)=x(P)-\alpha(P)$ from which we get $x(P)=\alpha$. 
	Now, let $Q$ be a place of $F'$ lying over $P$. Then, we have
	$v_Q(x-\alpha)=e(Q | P)\,v_P(x-\alpha) > 0$.
	Therefore, proceeding as before we get $x(Q)=\alpha=x(P)$ as desired. 
\end{proof}

Despite the trivial case of the previous example, we will show now that in finite Galois extensions of function fields, the conorm lift of (certain) cyclic AG-code is also cyclic.

\begin{thm} \label{teo conorm}
Let $F'/F$ be a Galois extension of algebraic function fields over $\ff_q$ of degree $m$, with $(m,q)=1$  or $q \mid m$.
Let $\mathcal{C} = C_{\rr}^F(D,G)$ be an AG-code such that every place in $\sop (D)$ is totally ramified in $F'$. 
Then, $\cod$ is cyclic if and only if $\cod' = \mathrm{Con}_{F'/F}(\cod)$ is cyclic.
\end{thm}

\begin{proof}
By Lemma \ref{lemin subcod} we know that $\cod \subseteq \cod'$ and hence $\cod$ is cyclic if $\cod'$ is cyclic.

Now, assume that $\cod=C_\rr^F(D,G)$ is cyclic. We want to show that $\cod'= C_\rr^{F'}(D',G')$ is cyclic too. 
Suppose that $D=P_1+\cdots+P_n$ where the $P_i$'s are different rational places totally ramified in $F'$. For every $i=1,\ldots,n$, let 
$P_i'$ be the only place in $F'$ above $P_i$. Thus, we have
$D' = P_1'+\dots+P_n'$.

Hence, given a codeword $c'=(x'(P_1'),\ldots,x'(P_n')) \in \cod'$, where $x'\in \rr(G')$, we want to show that we can find an element 
$z'=z'(x') \in \rr(G')$ satisfying \eqref{cyc cond}, i.e.\@
\begin{equation} \label{cyc cond2}
(z'(P_1'),z'(P_2'), \ldots, z'(P_n')) = (x'(P_n'),x'(P_1'), \ldots,x'(P_{n-1}')).
\end{equation}
 Let us first assume that $(m,q)=1$. Let $\mathcal{G}=\gal(F'/F)$ and consider the element  
\begin{equation} \label{x gal}
x:=\tr(x') = \sum_{\sigma \in \mathcal{G}} \sigma(x') \: \in F .
\end{equation}
Since $\tr(\rr(G')) = \rr(G)$, by Corollary 
\ref{coro1conorm}, we have that $x \in \rr(G)$, actually.
Since $\cod$ is cyclic, by \eqref{cyc cond} there is an element $z\in \rr(G)$ such that
\begin{equation} \label{C cyc} 
(z(P_1),z(P_2), \ldots, z(P_n)) = (x(P_n),x(P_1), \ldots,x(P_{n-1})).
\end{equation}

Now, put $z'=m^{-1}z \in \rr(G')$ where $m^{-1}$ is the inverse modulo $p=\mathrm{char}(\ff_q)$. 
Hence for every $i$ mod $n$ we have
$$ z'(P_i') = m^{-1} z(P_i') = m^{-1} z(P_i) = m^{-1}x(P_{i-1}) = m^{-1} x(P_i')$$
where we have used \eqref{C cyc} and Lemma \ref{lemitin}.
Thus, by \eqref{x gal} we get
\begin{align*}
   z'(P_i') & = m^{-1}\Big( \sum_{\sigma \in \mathcal{G}}\sigma(x') \Big) (P_{i-1}') 
	= m^{-1}\sum_{\sigma \in \mathcal{G}}x'(\sigma^{-1}(P_{i-1}')) \\
     &  = m^{-1}\sum_{\sigma \in \mathcal{G}}x'(P_{i-1}') 
     = m^{-1} m \, x'(P_{i-1}') = x'(P_{i-1}') \,.
 \end{align*}  
In this way, we see that $z'\in \rr(G')$ satisfies \eqref{cyc cond2} and therefore $\cod'$ is a cyclic code as desired.

In the case $q \mid m$ we use a similar argument but now considering 
 $$x:=N(x')=\prod_{\sigma \in \mathcal{G}} \sigma(x') \: \in F ,$$
where $\mathcal{G}=\gal(F'/F)$. Now, we have that $N(\rr(G'))\subseteq \rr(G)$ by Remark \ref{obs norma} so that $x\in \rr(G)$. We define $z'=z$. For every $i$ mod $n$, we have
\begin{align*}
   z'(P_i') = z(P_i') = x(P_{i-1}') & = \Big( \prod_{\sigma \in \mathcal{G}}\sigma(x') \Big) (P_{i-1}') = \prod_{\sigma \in \mathcal{G}}x'(\sigma^{-1}(P_{i-1}')) \\
      & = \prod_{\sigma \in \mathcal{G}}x'(P_{i-1}') = (x'(P_{i-1}'))^m = x'(P_{i-1}') \,.
 \end{align*}
In this way we see that $z'\in \rr(G')$ satisfies \eqref{cyc cond2} and therefore $\cod'$ is a cyclic code, as desired.
\end{proof}

We have seen in Lemma \ref{lemin subcod} that under certain conditions, the original code $\cod$ is a subcode of its conorm lift $\cod'$.
If, in addition, the code $\cod$ is cyclic, then both codes coincide.  That is, as algebraic codes over $\ff_q$, $\cod$ 
and $\cod'$ are the same code. 
\begin{coro} \label{C'=C}
Under the same hypothesis of Theorem \ref{teo conorm}, if $\cod$ is cyclic then $\cod=\cod'$. 
\end{coro}

\begin{proof}
We know by Lemma \ref{lemin subcod} that $\cod\subseteq\cod'$. 
Let $c' \in \cod'$. Then $c'=(x'(P_1'),\ldots,x'(P_n'))$ with $x'\in \rr(G')$. By Theorem \ref{teo conorm}, $\cod'$ is 
cyclic and hence the cyclic shift $s(c')\in \cod'$. Thus, there is $z'\in \rr(G')$ satisfying the cyclic condition
$$z'(P_1')=x(P_n'), \quad z'(P_2')=x(P_1'), \quad \ldots, \quad z'(P_n') = x(P_{n-1}').$$ 
In the proof of Theorem \ref{teo conorm} we showed how to construct this element $z'$ performing the shift, and that $z'$ is actually in $\rr(G)$, by construction. Therefore, $s(c')\in \cod$ and, in this way, $c'=s^n(c') \in \cod$. This implies that $\cod' \subseteq \cod$ and thus 
$\cod=\cod'$.  
\end{proof}

In other words, given a cyclic AG-code $\cod=C_\rr(D,G)$ over a finite Galois extension $F'/F$ of degree $m$, such that either $m$ is coprime to $q$ or $q$ divides $m$, and where the support of $D$ is totally ramified, the conorm lift gives a geometric representation $\cod'=\con (\cod)$ of $\cod$ in a function field of greater genus.

\section{Classic codes as conorm codes}
In this final section we show that repetition codes, Hermitian codes and Reed-Solomon codes can be considered, in many general cases, as conorm codes of rational AG-codes.

\subsubsection*{Repetition codes} 
Any repetition code 
	$$\mathcal{R}_q(n)=\{(c,\ldots,c) : c\in \ff_q\}$$
	of length $n\le q+1$ can be represented as a 
	rational AG-code in $F=\ff_q(x)$ as 
	$\cod=C_\rr(D,(y))$, with $ D=P_1+\cdots+P_n$, 
	where $P_1,\ldots,P_n$ are rational places of $F$ and $(y)$ is any principal divisor 
	disjoint with $D$. In fact, if $c\in \cod$, then 
	$$c=(x(P_1),\ldots,x(P_n))=(x,\ldots,x),$$ 
	since $x\in \rr((y))=\ff_q$ ($\deg G=0$ implies $\dim \rr(G)=1$) and hence, 
\begin{equation} \label{rep}
\mathcal{R}_q(n)=\cod=C_\rr(D,(y)).
\end{equation}
Furthermore, we have the following.

\begin{lem}
	In the case of geometric extensions, the conorm code of a repetition code is a repetition code. 
\end{lem}

\begin{proof}
	Consider the repetition code $\mathcal{R}_q(n)$ as in \eqref{rep} defined over a rational function field $F=\ff_q(x)$.
	Suppose that $F'/F$ is a geometric extension over $\ff_q$ of degree $m$ and genus $g'>0$.
	Since $\con_{F'/F}((y)^F)=(y)^{F'}$, we have
	$$\cod' = \con_{F'/F} \big( C_\rr^F(D,(y)^F) \big) = C_\rr^{F'} \big(D', (y)^{F'} \big),$$
	with $D' = D_1+\cdots+D_n$.
	But $\cod'$ is a repetition code by the previous comments and thus we get 
	$$\cod'=\mathcal{R}_q(n'),$$
	with $n'$ as in \eqref{n'}, as we wanted to show. 
\end{proof}

By using Kummer extensions we can give a partial converse of the previous result.
In fact, we will show that any repetition code is the conorm lift to a Kummer extension of the field $\ff_q$ viewed as a rational AG-code. 

\begin{prop} \label{prop rep}
	If $n \mid q-1$, the repetition code $\mathcal{R}_q(n)$ is a conorm code. 
\end{prop}

\begin{proof}
	Consider the rational function field $F=\ff_q(x)$ and let $F'=F(y)$ be the Kummer extension of $F$ given by
	$$y^n = (x-\alpha)(x-\alpha^{-1})$$
	where $n\mid q-1$, $\alpha\in \ff_q^*$ and $\alpha\ne \alpha^{-1}$.
	
	By Proposition 6.3.1 in \cite{Sti} we have that $F'/F$ is cyclic of degree $n$ and $\ff_q$ is the full constant field of $F'$ whose genus is
	$g = [\tfrac{n-1}{2}]$. Also, the places 
	$P_{\alpha}$ and $P_{\alpha^{-1}}$, the zeroes of $x-\alpha$ and $x-\alpha^{-1}$ respectively, are totally ramified in $F'/F$.
	
	Let $\varphi(T) = T^n - (x-\alpha)(x-\alpha^{-1}) \in \ff_q[T]$ and let $\bar{\varphi}(T)$ be its reduction mod $P_0$, the zero of $x$ in $F$. Since
	$x(P_0)=0$ and $n\mid q-1$ then
	$$\bar{\varphi}(T)= T^n -1 = \prod_{i=1}^n(T-a_i) \in \ff_q[T] \,. $$
	Therefore, by Kummer Theorem, $P_0$ splits completely in $F$.
	
	Let $D=P_1+\cdots + P_n$, where $P_1,\ldots,P_n$ are the (rational) places of $F'$ lying over $P_0$ and $G=(z)^{F'}$ with $z\in F$. 
	By the previous example we have 
	$$\cod' := C_\rr(D,(z)^{F'})=\mathcal{R}_q(n)$$ and 
	clearly $\cod'= \con_{F'/F} (\cod)$ where $\cod=C_\rr(P_0, (z)^F) = \ff_q$.
\end{proof}

\subsubsection*{Hermitian codes} 
We now show that certain Hermitian codes are conorm codes over Hermitian function fields of rational AG-codes.

Consider the Hermitian function field $H=\ff_{q^2}(x,y)$ as the degree $q$ extension of the rational function field $F=\ff_{q^2}(x)$, given by the equation 
$$y^q+y=x^{q+1}.$$
The field $F$ has $q^2+1$ rational places $P_1,\ldots,P_{q^2}$ and $P_\infty$, the pole of $x$. For each $\alpha\in \ff_{q^2}$ there are $q$ 
elements $\beta \in \ff_{q^2}$ satisfying 
$$\beta^q+\beta = \alpha^{q+1},$$ 
and for all such pairs $(\alpha,\beta)$ there is a unique place 
$P_{\alpha,\beta}$ in $H$ such that $x(P_{\alpha,\beta})=\alpha$ and $y(P_{\alpha,\beta})=\beta$. Thus, $H$ has $q^3+1$ rational places, 
$q$ places over each rational place $P_{\alpha,\beta}$ of $F$ and $Q_\infty$, the common pole of $x$ and $y$ in $H$,  lying over $P_\infty$. 

Hermitian codes are the 1-point AG-codes defined as
\begin{equation} \label{Hr}
\mathcal{H}_a = C_\rr^H(D', aQ_\infty),
\end{equation}
where 
\begin{equation} \label{D'}
D' = \sum_{P\in \mathbb{P}(H)\smallsetminus \{Q_\infty\}} P =\sum_{\beta^q+\beta = \alpha^{q+1}} P_{\alpha,\beta}.
\end{equation}

\begin{rem}
In order to construct codes from the Hermitian function field in which the divisor $G$ is not a one-point divisor, one can consider the function field $H$ not as an Artin-Schreier extension of $\ff_q^2(x)$, but rather as a Kummer extension of $F=\ff_q^2(y)$ defining the divisor $G$ by means of the zeroes of $y^q+y$ and $D$ by the places which are unramified over $F$ (see Example 4.8 in \cite{BQZ18}).
\end{rem}

\begin{prop}
If $q\mid a$ then $\mathcal{H}_a$ is a conorm code.
\end{prop}

\begin{proof}
Using the above notation, let us consider the code 
$$\cod_t = C_\rr^F(D, t P_\infty),$$ 
where $D = P_1+\cdots+P_{q^2}$.
Note that if $q\mid a$ then 
	$$\mathcal{H}_a = \con_{H/F} (\cod_{a/q})$$ 
where $H=\ff_{q^2}(x,y)$ is the Hermitian field given by $y^q+y=x^{q+1}$, because we have $\con (D)= D'$ as in \eqref{D'} and $\con (sP_\infty) = sqQ_\infty$ for any $s$.
\end{proof}

As an application of these results, we  show  next  that the identity \eqref{con dual} fails to hold in general, that is 
the dual of a the conorm code of $\cod$ is not necessarily the conorm of the dual code of $\cod$. 

\begin{exam}\label{counterexample}
Consider the Hermitian function field $H = \ff_{q^2}(x,y)$, i.e.\@ the extension of the rational function field 
	$F=\ff_{q^2}(x)$ defined by $y^q+y=x^{q+1}$. Consider the rational AG-code $\mathcal{C}_q=C_{\rr}^F(D,qP_\infty)$, where $D$ is the sum of all the rational places of $F$ different from $P_\infty$, with parameters $[q^2, q+1]$ and let $\mathcal{H}_a = C_{\rr}^H(D', aQ_\infty)$ be the Hermitian code, where $D'$ is the sum of all the places over the support of $D$, with parameters $[q^3, k']$. 
	
	 We have seen that $\mathcal{H}_a$ is the conorm code of $C_q$ if $q$ divides $a$. Thus we can take, for instance,  $a=3q^2$. By ($b$) of Proposition 8.3.3 of \cite{Sti}, the code $\mathcal{H}_{3q^2}$ has dimension 
	$$k' = \dim(\mathcal{H}_{3q^2}) = 3q^2+1 - \tfrac12 q(q-1) = \tfrac{5q^2+q}{2}+1.$$ 
	
	Now, on the one hand, since $\mathcal{H}_a^\perp = \mathcal{H}_{q^3+q^2-q-2-a}$ for any $a$, in our case we have
	$$\con(\cod_q)^\perp  = \mathcal{H}_{3q^2}^\perp = \mathcal{H}_{q^3-2q^2-q-2}$$
with parameters $[q^3,k^\perp]$. It is clear that 
$$k^\perp =q^3 - k' = q^3-\tfrac{5q^2+q+2}{2}>0,$$	 
for any $q\ge 3$.	

On the other hand, $\cod_q^\perp=C_\rr^F(D,G^\perp)$ with $G^\perp = D-qP_\infty + (\eta)$  is a $[q^2,q^2-q-1]$-code, where $\eta$ is a Weil differential. The conorm code is $\con(\cod_q^\perp)$ with parameters $[q^3, \tilde k]$.
By \eqref{k'} of Proposition \ref{prop C'} we have 
\begin{eqnarray*}
\tilde k & \ge & q(\deg G^\perp +1-g)-\tfrac 12 \deg \dif(H/F) \\ 
&= &  q(q^2-q-2) -\tfrac 12 (q^2-q) =  q^3-\tfrac{3q^2-3q}{2} = q^3- \tfrac 32 q(q+1) >0,
\end{eqnarray*}
for any $q\ge 3$.

It is straightforward to check that $k^\perp < \tilde k$ if and only if $q^2+1>q$ and this last inequality holds for every $q$.
Since the dimensions of the codes $\con(\cod_q)^\perp$ and $\con(\cod_q^\perp)$ are different, we have that 
$$\con(\cod_q)^\perp \ne \con(\cod_q^\perp),$$ 
as we wanted to show. For instance, if $q=3$ the parameters of  
the codes $\con(\cod_3)^\perp$ and $\con(\cod_3^\perp)$ over $H=\ff_9(x,y)$, $y^3+y=x^4$, are $[27,2]$ and $[27,\ge 9]$, respectively. Also, the parameters of the codes $\con(\cod_4)^\perp$ and $\con(\cod_4^\perp)$ over $H=\ff_{16}(x,y)$ with $y^4+y=x^5$ are $[64,21]$ and $[64,\ge 34]$, respectively.
\end{exam}

\subsubsection*{Reed-Solomon codes} 
As a final application, we show that classical Reed-Solomon codes can be obtained as conorm codes of rational cyclic AG-codes.

\begin{prop}
Any Reed-Solomon code is a conorm code. 
\end{prop}

\begin{proof}
Let $n=q-1$ and $k$ be such that  $1\leq k \leq n$ and let us consider the Reed-Solomon code
$$\cod_k=\{(f(\beta), f(\beta^2),\ldots, f(\beta^n)): f \in \rr_k)\} $$ 
over $\ff_q$, where $\beta\in \ff_q$ is a primitive element of the subgroup $\ff_q^\ast$ and $$\rr_k=\{f\in \ff_q[x]:\deg f\leq k-1\}.$$ 

The code $\cod_k$ can be represented as a rational AG-code as follows. 
Let $F=\ff_q(x)$ be a rational function field and denote by $P_i$ the zero of $x-\beta^i$ in $F$ for $i=1,\ldots, n,$ and by $P_\infty$ the pole of $x$ in $F$.  Let $u\in F$ be such that $u(P_i)=1$ for $i=1, \ldots,n$ (such an element exists by the Approximation theorem). Now, letting 
$$D=P_1+\cdots+P_n \qquad \text{and} \qquad G=(k-1)P_\infty+(u),$$ 
where $(u)$ denotes the principal divisor of $u$ in $F$, we have (see Proposition 2.3.5 of \cite{Sti}) 
that $\cod_k=\cod_\rr(D,G)$.

Let us consider now the field extension $F'/F=\ff_q(x,y)/\ff_q(x)$ where $x$ and $y$ satisfy 
$$y^n = (x-\beta)(x-\beta^2) \cdots (x-\beta^n).$$ 
By Proposition 6.3.1 in \cite{Sti} we have that $F'/F$ is a cyclic extension of degree $n$ and the places $P_1,\ldots, P_n$ are totally ramified in $F'/F$. In this way, we are in the hypothesis of Theorem \ref{teo conorm} and Corollary \ref{C'=C} and thus $$\cod'=\con(\cod_k)=\cod_\rr(\con(D), \con(G))$$ 
satisfies $\cod'=\cod_k$, as we wanted to show.
\end{proof}

\section*{Final Remarks}
\subsubsection*{Generalized conorm codes}
The definition of conorm code given in Section 2 can be generalized as follows. 
Let $F'/F$ be a finite extension of function fields over $\ff_q$ and let $\cod = C_\rr^F(D,G)$ be an AG-code as in \eqref{CDG}. Suppose   that every place $P'$ of $F'$ over any place $P\in\sop(D)$  is rational. We define the \textit{generalized conorm code} associated to $\cod$, or just the \textit{conorm of $\cod$}, as
\begin{equation} \label{Con C}
\cod' = \con_{F'/F} (\cod) = C_\rr^{F'}(D',G'), 
\end{equation}
where
$$D' =\sum_{i=1}^n \sum_{Q|P_i} Q
\qquad \text{ and } \qquad G'=\con_{F'/F} (G).$$

\sk 

\noindent \textit{Note:} For any given AG-code $\cod = C_\rr^F(D,G)$ over a function field $F/\ff_q$, we can always find a field extension $F'$ of $F$ such that the condition on the rationality of the places above the places in the support of $D$ holds: we just take a suitable constant field extension $F'=F\ff_{q^t}$ of $F$.

\sk 
This construction of generalized conorm codes behaves well on finite towers of function fields. 
That is, given $F\subset F' \subset F''$ and $\cod=C_\rr^F(D,G)$ we have
\begin{equation}\label{wellbhv}
\con_{F''/F}(\cod) = \con_{F''/F'} (\con_{F'/F} (\cod)),
\end{equation}
or, in other words, if $\cod' = \con_{F'/F}(\cod)$, then
\begin{equation*}
\cod'' = \con_{F''/F'} (\cod') = \con_{F''/F}(\cod).
\end{equation*}
This is a direct consequence of the fact that 
$$\con_{F''/F}(G) = \con_{F''/F'} (\con_{F'/F} (G)),$$ 
for any divisor $G$ in $F$ and the fact that if every place $R$ in $F''$ over a place $P$ in the support of $D$ is rational, then every place $Q$ in $F'$ over a place $P$ in the support of $D$ is also rational. Furthermore
$$\bigcup_{P\,\in\,\sop{D}}\{R \mid P : R \in \pl(F'') \} = 
\bigcup_{Q\,\in \,\sop{D'}}\{R \mid Q: R \in \pl(F'') \}.$$

Notice that  with the definition of conorm codes given in Section 2, the equality \eqref{wellbhv} was obtained in the particular cases of complete splitting or total ramification of the places in $\sop(D)$. Moreover when every place in $\sop(D)$ splits completely in $F'$, we have that $D'$ is actually the conorm of $D$.

\subsubsection*{Code-based cryptography}

Different families of codes have proved to be insecure for code-based cryptography and there have been many attempts to replace traditional Goppa codes. In \cite{JM96}, Janwa and Moreno proposed to use a collection of AG-codes on curves for the McEliece cryptosystem, but this was broken for codes on curves of genus $g \leq 2$ by Faure and Minder (\cite{FM08}). 

Couvreur, M\'arquez-Corbella, Mart\'inez-Moro, Pellikaan and Ruano (\cite{CMP14} and \cite{MMPR}) have managed to break certain higher-genus cryptosystems based on evaluation codes, but none of these attacks are against subfield subcodes. The security status of the McEliece public key cryptosystem using algebraic geometry codes is, to the best of our knowledge,  still not completely settled and remains as an open problem.

The construction of taking conorm codes can be iterated, so in principle it can be applied to a code defined over the base field of a tower of function fields. By the results in Theorem \ref{teo conorm} and Corollary \ref{C'=C} we can begin with a cyclic AG-code defined over the rational function field ($g=0$). Under appropriate conditions, we get the same code (as a conorm code) defined on a bigger field with greater genus. The procedure can  be repeated in such a manner to get the same algebraic cyclic code defined on a function field of genus arbitrarily large. Can this procedure be useful in some way in code-based cryptography? This is a question we hope to answer in a forthcoming work.

\bibliographystyle{plain}

\end{document}